\documentclass[11pt]{article}

\date{24th September 2013}
\def\titlename{ZL-amenability and characters for the restricted direct products of finite groups}
\def\authname{M. Alaghmandan, Y. Choi, and E. Samei}

\title{{\sc \titlename}}
\author{{\sc \authname}}

\def\contact{
\noindent Mahmood Alaghmandan, Yemon Choi, and Ebrahim Samei

\noindent
Department of Mathematics and Statistics, University of Saskatchewan\\
McLean Hall, 106 Wiggins Road, Saskatoon (SK), Canada S7N 5E6.

\noindent
E-mails: \texttt{m.alaghmandan@usask.ca}, \texttt{choi@math.usask.ca}, \texttt{samei@math.usask.ca} \\
}

\usepackage{amsmath,amssymb}
\newcounter{pulse}[section]
\numberwithin{pulse}{section}
\numberwithin{equation}{section}

\usepackage{amsthm}
\newcommand{\tf}{\sc}

\newtheorem{thm}[pulse]{\tf Theorem}

\newtheorem{lem}[pulse]{\tf Lemma}
\newtheorem{cor}[pulse]{\tf Corollary}

\theoremstyle{definition}
\newtheorem{dfn}[pulse]{\tf Definition}
\newtheorem{eg}[pulse]{\tf Example}
\newtheorem{rem}[pulse]{\tf Remark}

\newcommand{\para}[1]{\paragraph{#1}}

\newenvironment{numlist}{%
\begin{enumerate}

}{\end{enumerate}\ignorespacesafterend}

\newcommand{\dt}[1]{{\it#1}\/}  % usual journal style; can change to YC preference if so desired
\newcommand{\st}{\mathbin{\colon}}
\newcommand{\defeq}{:=}  % alias which we may wish to change in submission

\newcommand{\Aff}{\operatorname{Aff}}  % affine group of a ring
\newcommand{\Ch}{\operatorname{Ch}}

\newcommand{\Irr}{\operatorname{Irr}}
\newcommand{\Sp}{\operatorname{sp}} % spectrum
\newcommand{\supp}{\operatorname{supp}}

\newcommand{\Tr}{\operatorname{Tr}} % unnormalized trace
\renewcommand{\ker}{\operatorname{Ker}}

\newcommand{\AM}{\operatorname{AM}}
\newcommand{\Zl}{\operatorname{Z\ell}}
\newcommand{\mcr}{\operatorname{mcr}}  % late addition, will eventually put in main macro file

\newcommand{\Proj}{{\bf P}} % can't decide on letter for canonical projection HMs

\newcommand{\Ind}{{\bf I}} % index set

\newcommand{\rdp}[3]{{\bigoplus_{#1\in#2} {#3}_{#1}}}

\newcommand{\der}[1]{{#1}^\prime}  % notation for derived subgroup.
\newcommand{\affine}[2]{\left(\begin{matrix} #1 & #2 \\ 0 & 1 \end{matrix}\right)} % for matrices in affine group

\newcommand{\groupprop}{AIC}
\newcommand{\charprop}{absolutely idempotent}
\newcommand{\St}{\operatorname{St}} % Steinberg character

\newcommand{\abs}[1]{\vert #1 \vert}
\newcommand{\norm}[1]{\Vert #1 \Vert}
\newcommand{\pair}[2]{\langle#1,#2\rangle}
\newcommand{\ptp}{\mathbin{\widehat{\otimes}}}
\newcommand{\bigptp}{\mathop{\widehat{\bigotimes}}\nolimits}

\newcommand{\Cplx}{{\mathbb C}} % YC alias, by habit
\newcommand{\Nat}{{\mathbb N}}
\newcommand{\Cst}{\operatorname{{\rm C}^*}}
\newcommand{\SL}{\operatorname{SL}}
\newcommand{\PSL}{\operatorname{PSL}}
\newcommand{\FCbar}{\ensuremath{{\rm[FC]}^{-}}}

\newcommand{\veps}{\varepsilon}
\newcommand{\om}{\omega} % YC alias (force of habit)
\newcommand{\tilphi}{\widetilde{\varphi}}
\newcommand{\tilom}{\widetilde{\omega}}

\newcommand{\cA}{{\mathcal A}}
\newcommand{\cB}{{\mathcal B}}
\newcommand{\cC}{{\mathcal C}}

\newcommand{\bbF}{{\mathbb F}}

\newcommand{\bZ}{\operatorname{\bf Z}}

\newcounter{point}
\newcommand{\TOFLAG}[1][]{\refstepcounter{point}\marginpar{\fbox{\textcolor{Red}{\large\bf{\textsf C\thepoint}}}}}

\usepackage[usenames]{color}
\renewcommand{\dt}[1]{\textcolor{Bittersweet}{\textsf{#1}}}

\renewcommand{\para}[1]{\paragraph{#1.}}

\usepackage{hyperref}

\usepackage{sectsty}
\allsectionsfont{\sc}

\begin{document}

\maketitle

\begin{abstract}
Let $G$ be a restricted direct product of finite groups $\{ G_i \}_{i\in I}$, and let $\Zl^1(G)$ denote the centre of its group algebra.
We show that $\Zl^1(G)$ is amenable if and only if $G_i$ is abelian for all but finitely many $i$, and characterize the maximal ideals of $\Zl^1(G)$ which have bounded approximate identities.
We also study when an algebra character of $\Zl^1(G)$ belongs
to $c_0$ or $\ell^p$ and provide a variety of examples.


\end{abstract}

\begin{section}{Introduction}
The $L^1$-convolution algebra of a locally compact group $G$ is amenable if and only if the group $G$ is amenable.
In contrast, there is as yet no `intrinsic' characterization of those groups $G$ for which $\operatorname{ZL}^1(G)$, the centre of $L^1(G)$, is amenable. It will be convenient to call such groups \dt{ZL-amenable}.

The article \cite{AzSaSp} studied the problem of which compact groups are ZL-amenable.
One can consider the corresponding problem for (discrete) FC groups, that is, the groups in which each conjugacy class is finite.
We note that the problem is solved for \emph{finitely generated} FC groups:
for, by a result of B. H. Neumann \cite[Theorem 5.1]{Neumann_FCgroup}, each such group has a finite derived subgroup; and when a discrete group has a finite derived subgroup, then it is ZL-amenable by a result of Stegmeir \cite[Theorem 1]{Steg}.

The present paper is concerned with a particular class of infinitely generated FC groups, namely the restricted direct products of finite groups (\dt{RDPF groups}, for short) whose formal definition is given below in Definition~\ref{dfn:rdp}.
The $\Cst$ and von Neumann algebras of discrete RDPF groups were studied in some old work of Mautner \cite{Mautner_RDPF}, but to our knowledge the centres of their $\ell^1$-group algebras have not been studied explicitly.

When $G$ is such a group, we are able to study $\Zl^1(G)$ in some detail. We show in Theorem~\ref{t:amen-for-rdp} that an RDPF group is ZL-amenable if and only if its derived subgroup is finite.
 Then, in Theorem~\ref{t:every-max-ideal-has-BAI},
we obtain a characterization of those RDPF groups $G$ for which every maximal ideal in $\Zl^1(G)$ has a bounded approximate identity.
(For a commutative, unital Banach algebra, the condition that each maximal ideal has a bounded approximate identity can be thought of as a weaker version of amenability. It has been rediscovered and studied in further depth under the name of \dt{character amenability}, see e.g. \cite{KanLauPym,mm} for further details, not restricted to the unital or commutative cases.)

Using our techniques, in Section~\ref{s:c0-lp}, we revisit the main counter-example from Stegmeir's paper ({\it op cit.}\/), and are able to give simpler proofs of some of his results.
In doing so, we are led to consider when given characters on $\Zl^1(G)$ belong to $c_0$, when viewed as functions on $G$ in the natural way. We obtain a necessary and sufficient condition for this to occur, and also study the related question of when such characters lie in $\ell^p$ for various exponents $p\in [1,\infty)$.
\end{section}

\begin{section}{Notation and other preliminaries}\label{s:prelim}
%% \YCrem{Unnormalized characters $\chi_\pi=\Tr\pi$ seem more natural from the point of view of character theory on finite groups: and most of what we do is using the language of finite groups rather than the general FC group theory of Kaniuth, Mosak et~al.}
In this section we set up some notation and some preliminary results that will be needed. Most of these results are well-known to specialists in non-abelian harmonic analysis but we have repeated them, often in special cases, to keep the present paper more self-contained.

\begin{subsection}{Important conventions}
Since we are interested in discrete groups, we will always equip finite groups with counting measure, not the uniform probability measure. Occasionally this will mean that in quoting results concerning harmonic analysis on compact groups, we have to insert a scaling factor. It should be clear, from context, when and how this is done.

\para{Infinite products and sums over arbitrary indexing sets}
In several places we will want to consider certain infinite sums or products where the numbers involved are indexed by some arbitrary set~$\Ind$. Since we only need to consider sums of positive numbers and products of numbers~$\geq 1$, we shall interpret this as unconditional summation
\[ \sum_{i\in\Ind} a_i \defeq \sup\left\{ \sum_{i\in F} a_i \st F\subseteq\Ind, |F|<\infty\right\} \]
when $a_i\geq 0$ for all $i\in\Ind$; and likewise we define
\[ \prod_{i\in\Ind} b_i \defeq \sup\left\{ \prod_{i\in F} b_i \st F\subseteq\Ind, |F|<\infty\right\} \]
when $b_i\geq 1$ for all $i$.

\para{Algebra characters and group characters}
The Gelfand spectrum of a commutative, unital Banach algebra $A$ will be denoted by $\Sp(A)$; usually we shall think of it as the space of non-zero multiplicative linear functionals, rather than the maximal ideal space.

Unfortunately, the word `character' is used in both Banach algebra theory and in the representation theory of finite groups, and means two slightly different things. In this article, to try and prevent ambiguity, we will always use the phrase \dt{algebra character} to mean a character in the sense of Gelfand theory, i.e.~a non-zero multiplicative linear function from a complex Banach algebra to the ground field; and we will always use the phrase \dt{group character} to mean a character in the sense of finite group theory, i.e.~the trace of a finite-dimensional (unitary) representation.

Given a group character $\chi$, we say that \dt{$\pi$ affords $\chi$} when $\pi$ is a finite-dimensional representation of $G$ whose trace is $\chi$.
The \dt{degree of $\chi$}, which we denote by $d_\chi$\/, is defined to be the dimension of any $\pi$ which affords $\chi$; equivalently, $d_\chi=\chi(e)$ where $e$ is the identity of the group in question.

\begin{rem}
When viewing a group character as the trace of a suitable representation, we use the unnormalized trace, denoted by $\Tr$, so that the trace of the $d\times d$ identity matrix is~$d$. (In some sources, one normalizes group characters so that they each take the value $1$ on the identity element of the group; our convention is more in keeping with that used in finite group theory.)
\end{rem}

\para{Notation}
Recall that if $G$ is a (discrete) group, its \dt{derived subgroup} or \dt{commutator subgroup} is the normal subgroup of $G$ generated by the set of all commutators of elements in $G$. We shall denote the derived subgroup of $G$ by~$\der{G}$.

\end{subsection}

\begin{subsection}{General properties of $\Zl^1$}
It is well known that when $G$ is finite, algebra characters on $\Zl^1(G)$ correspond to irreducible group characters of $G$. More precisely, we have the following result.

\begin{lem}\label{l:spectrum-for-finite-groups}
Let $G$ be a finite group. If $\psi$ is an algebra character on $\Zl^1(G)$, then there is a unique irreducible group character $\chi$ of $G$
which satisfies
\begin{equation}\label{eq:mlf-via-irrep}
\psi(f) = \sum_{x\in G} f(x) {d_\chi}^{-1} \chi(x^{-1})
\qquad\text{for all $f\in\Zl^1(G)$.}
\end{equation}
Conversely, for each irreducible group character $\chi$ on $G$, the formula \eqref{eq:mlf-via-irrep} defines an algebra character on $\Zl^1(G)$.
\end{lem}

This result is well known and has been generalized to the setting of locally compact \FCbar\ groups (in particular, it is a very special case of~\cite[Corollary 1.3]{LiMos}.)
Nevertheless, we shall give a self-contained outline of the proof of Lemma~\ref{l:spectrum-for-finite-groups}, since this is an instructive warm-up for the arguments of Section~\ref{s:BAI-in-maxideal}, and may help to make that section more accessible.

\begin{proof}[Sketch of the proof]
By standard properties of the nonabelian Fourier transform for finite groups, $\Zl^1(G)$ is isomorphic as a Banach algebra to $\Cplx^{\widehat{G}}$. Hence it is spanned by its minimal idempotents, which are all of the form $|G|^{-1}d_\chi\chi$ for some irreducible group character~$\chi$.

Suppose $\psi$ is an algebra character on $\Zl^1(G)$. If we write
$f = \sum\nolimits_\sigma a_\sigma |G|^{-1} d_\sigma \sigma$
for appropriate complex numbers $a_\sigma$, then there exists a unique $\chi$ such that $\psi(f)=a_\chi$. By Schur orthogonality, $a_\chi = \sum_{x\in G} f(x) {d_\chi}^{-1} \chi(x^{-1})$, as required.

Conversely, if $\chi$ is a given irreducible group character on $G$, then it is easily checked that
defining $\psi$ by \eqref{eq:mlf-via-irrep} gives an algebra character on~$\Zl^1(G)$.
\end{proof}

At certain points below we will rely crucially on a result of Rider concerning norms of central idempotents in the group algebra of a compact group.
(For our applications, we only need the case of finite groups; but this restriction does not seem to make the result significantly easier to prove.)

\begin{thm}[Rider; {see \cite[Lemma 5.2]{ri}}]\label{t:rider}
Let $G$ be a compact group, $\lambda$ a Haar measure on it, and $\psi$ a finite linear combination of irreducible group characters on $G$. Suppose that $\psi*\psi=\psi$ and that
\[ \int_G \abs{\psi(x)}\,d\lambda(x) > 1. \]
Then
\[ \int_G \abs{\psi(x)}\,d\lambda(x) \geq \frac{301}{300} \,. \]
\end{thm}

\begin{rem}
Rider's result is stated for the case where $\lambda(G)=1$. However, a simple rescaling argument shows that this is equivalent to the formulation we have given.
We also note that $301/300$ is not sharp; it appears to be an open problem to determine the best constants here.
\end{rem}

The following lemma will be used later, in several places. We would like to thank the referee for indicating how our original proof could be made much simpler and shorter.

\begin{lem}\label{l:zl1-of-binary-product}
Let $H$ and $K$ be (discrete) FC groups. Then the canonical, isometric isomorphism of Banach algebras
\[ \theta: \ell^1(H)\ptp\ell^1(K) \to \ell^1(H\times K) \]
restricts to an isometric isomorphism of Banach algebras
\[ \Zl^1(H)\ptp\Zl^1(K) \cong \Zl^1(H\times K) \]
\end{lem}

\begin{proof}
The map $\theta$ is defined by setting $\theta(f\otimes g)(x,y) = f(x)g(y)$ for all $f\in\ell^1(H)$, $g\in \ell^1(K)$, $x\in H$ and $y\in K$. It is easy to see that $\theta(\Zl^1(H)\ptp\Zl^1(K))\subseteq \Zl^1(H\times K)$.

On the other hand, $\theta(\Zl^1(H)\ptp\Zl^1(K))$ is closed in $\ell^1(G\times H)$
and contains all indicator functions of conjugacy classes since
$1_{C_{(x,y)}}=1_{C_{x}}\times 1_{C_{y}}=\theta(1_{C_{x}}\otimes 1_{C_{y}})$. Such indicator functions have dense linear span in $\Zl^1(H\times K)$, so $\theta(\Zl^1(H)\ptp\Zl^1(K))\supseteq\Zl^1(H\times K)$.
\end{proof}

\end{subsection}

% 2013-08-09: moving section break
\end{section}

\begin{section}{The restricted direct product of a family of finite groups}
\label{s:rdp}

% YC tweak of ES rewrite
\begin{dfn}\label{dfn:rdp}
Let $\Ind$ be an indexing set and $(G_i)_{i\in\Ind}$ a family of groups; let $\prod_{i\in\Ind} G_i$ denote the set-theoretic product of these groups, which is itself a group in a canonical way. The \dt{restricted direct product} of the family $(G_i)$ is defined to be the subgroup
\[ \rdp{i}{\Ind}{ G} \defeq \left\{ (x_i)_{i\in\Ind} \in \prod_{i\in\Ind} G_i \st \text{$x_i=e_{G_i}$ for all but finitely many $i$} \right\}. \]
\end{dfn}

\begin{rem}
Our notation for the restricted direct product of groups is borrowed from the case where each of the groups $G_i$ is abelian: in this case, the restricted direct product is sometimes referred to as the \dt{direct sum} (since the restricted direct product gives the coproduct in the category of Abelian groups).
\end{rem}
% tweak by YC

The following remark assembles some basic properties of this construction. The proofs are all easy and are therefore omitted.
\begin{rem}\label{r:basic prop-RDPF}
% YC revert of ES edit
Let $(G_i)_{i\in\Ind}$ be a family of discrete groups and let $G=\rdp{i}{\Ind}{G}$.
\begin{numlist}
\item $Z(G) = \bigoplus_{i\in\Ind} Z(G_i)$.
\item\label{li:der-of-rdpf} $\der{G} = \bigoplus_{i\in\Ind} \der{G_i}$.
\item If each $G_i$ is nilpotent of class $n$, then so is $G$.
\item If each $G_i$ is solvable of length $n$, then so is $G$.
\end{numlist}
\end{rem}

\para{Inclusion and projection homomorphisms}
Let $(G_i)$ be a family of discrete groups, and let $F\subset\Ind$; write $F^c$ for $\Ind\setminus F$. Since
\[ \rdp{i}{\Ind}{G} \cong \bigl(\rdp{i}{F}{G}\bigr) \times \bigl(\rdp{i}{F^c}{G} \bigr) \]
by Lemma~\ref{l:zl1-of-binary-product} we obtain an isometric isomorphism of Banach algebras
\begin{equation}\label{tensor:center}
 \Zl^1(\rdp{i}{\Ind}{ G}) \cong \Zl^1(\bigoplus_{i\in F} G_i)\ptp\Zl^1(\bigoplus_{i\in F^c} G_i).
\end{equation}

Hence, if we write $E_F^c$ for the identity element of $\Zl^1(\bigoplus_{i\in F^c} G_i)$, there is a unital, isometric, homomorphism of Banach algebras
\begin{equation}\label{eq:inclusion-HM}
\imath_F :
\Zl^1(\bigoplus_{i\in F} G_i) \to
 \Zl^1(\rdp{i}{\Ind}{ G})
\end{equation}
defined by $\imath_F(a) = a\otimes E_F^c$. When $F$ is a singleton, say $\{j\}$, we denote $\imath_F$ by~$\imath_j$\/.
%% [{\bf I need it several times so I have added it here} -MA.]
If we denote by $\veps_{F^c}$ the augmentation character on $\Zl^1(\rdp{i}{F^c}{G})$, and denote by ${\rm id}_F$ the identity homomorphism on $\Zl^1(\rdp{i}{F}{G})$, then there is a unital, surjective homomorphism of Banach algebras
\begin{equation}\label{eq:projection-HM}
\Proj_F = {\rm id}_F \otimes \veps_{F^c} :
 \Zl^1(\rdp{i}{\Ind}{ G})
\to \Zl^1(\rdp{i}{F}{G})
\end{equation}
 which satisfies $\Proj_F(a\otimes b) = \veps_{F^c}(b) a$ for all $a\in \Zl^1(\rdp{i}{F}{G})$ and all $b\in\Zl^1(\rdp{i}{F^c}{G})$.

\bigskip
If $(G_i)_{i\in\Ind}$ is a family of \emph{finite} groups, we call $\rdp{i}{\Ind}{G}$ a restricted direct product of finite groups, or an \dt{RDPF group} for short. It is easy to show that every RDPF group is a discrete FC group.
Our philosophy in this paper is that the class of RDPF groups is very convenient from a technical point of view, and can be handled without needing much of the technical machinery for general FC or \FCbar\ groups. To illustrate this, we can characterize the RDPF groups which are ZL-amenable (the problem of characterizing all ZL-amenable, discrete FC groups remains open).

We recall that the notion of amenability for a given Banach algebra $\cA$ can be characterized in different ways, in particular through the existence of a so-called \dt{bounded approximate diagonal} for $\cA$. See, for instance, \cite[\S2.2]{RunLec} for the relevant definitions and proofs. This characterization allows us to quantify the amenability of a Banach algebra via its \emph{amenability constant}, which seems to have first been formally introduced in~\cite{am}, although the underlying ideas were tacitly recognized much earlier.

\begin{dfn}\label{d:AM(A)}
Let $\cA$ be a Banach algebra. The \dt{amenability constant of $\cA$}, which we denote by $\AM(\cA)$, is
\[
\inf \{\sup_{\alpha} \|\mu_\alpha\| \st \;(\mu_\alpha)_\alpha\;\text{ is a bounded approximate diagonal for}\; \cA\}
\]
where we define the infimum of the empty set to be $+\infty$.
\end{dfn}

Note that if $\cA$ is a unital Banach algebra, then $\AM(\cA)\geq 1$. It is a standard result in the theory of amenable Banach algebras that if $\cA$ and $\cB$ are Banach algebras, with $\cA$ amenable, and if $\phi: \cA\to \cB$ is a continuous homomorphism with dense range, then $\cB$ is amenable.
 See, for instance, \cite[Proposition 2.3.1]{RunLec}, whose proof makes it clear that
we furthermore have $\AM(\cB)\leq \norm{\phi}^2\AM(\cA)$.

As mentioned in the introduction, Stegmeier \cite{Steg} showed that if $\der{G}$ is finite then $G$ is ZL-amenable. One may ask if the converse holds, within the class of FC groups: that is, if $G$ is a discrete FC group which is ZL-amenable, is $\der{G}$ necessarily finite? The following result gives a positive answer in the special case where $G$ is an RDPF group.

\begin{thm}\label{t:amen-for-rdp}
Let $(G_i)_{i\in\Ind}$ be a family of finite groups and let $G=\rdp{i}{\Ind}{ G}$. Then the following are equivalent:
\begin{numlist}
\item\label{li:AM-i} $\Zl^1(G)$ is amenable;
\item\label{li:AM-ii} $G_i$ is abelian for all but finitely many $i$;
\item\label{li:AM-iii} $G$ is isomorphic to the product of a finite group with an abelian group;
\item\label{li:AM-iv} the derived subgroup of $G$ is finite.
\end{numlist}
\end{thm}

\begin{proof}
We start by defining $N=\{i\in\Ind \colon G_i \text{ is non-abelian}\}$.

\para{\ref{li:AM-i} $\implies$ \ref{li:AM-ii}}
As observed in \cite{AzSaSp}, it follows from Rider's result (Theorem~\ref{t:rider}) that $\AM(\Zl^1(H))\geq 1+1/300$ whenever $H$ is a finite nonabelian group.
Now suppose $\Zl^1(G)$ is amenable, and let $F$ be a finite subset of $N$.
Recall that we have a quotient homomorphism of Banach algebras
$\Proj_F: \Zl^1(G) \to \Zl^1\bigl(\rdp{i}{F}{G}\bigr)$,
as defined earlier in~\eqref{eq:projection-HM}. Therefore, by the remarks following Definition~\ref{d:AM(A)},
\[
\AM(\Zl^1(G))\geq \AM(\Zl^1(\rdp{i}{F}{G})).
\]
Moreover, it was proved in \cite{AzSaSp} that
\[ \AM(\Zl^1\bigl(\rdp{i}{F}{G}\bigr)) = \prod_{i\in F} \AM(\Zl^1(G_i)) \]
Hence
\[
\AM(\Zl^1(G))
\geq \prod_{i\in F} \AM(\Zl^1(G_i)) \geq (1+ 1/300)^{|F|} .
\]
Since $F$ was an arbitrary finite subset of $N$, this shows $N$ is finite.

\para{\ref{li:AM-ii} $\implies$ \ref{li:AM-iii}}
This is clear.

\para{\ref{li:AM-iii} $\implies$ \ref{li:AM-i}}
If $K$ is finite and $A$ is abelian, then by Lemma~\ref{l:zl1-of-binary-product},
\[ \Zl^1(K\times A) \cong \Zl^1(K)\ptp \Zl^1(A) = \Zl^1(K) \ptp\ell^1(A). \]
The right-hand side is the projective tensor product of two amenable Banach algebras, hence is amenable by standard hereditary properties of amenability. (See e.g.~\cite[\S2.3]{RunLec}.)

% YC rewrite
\para{\ref{li:AM-ii} $\iff$ \ref{li:AM-iv}}
We have $\der{G}=\bigoplus_{i\in I} \der{G_i}$ (as observed in Remark \ref{r:basic prop-RDPF}\ref{li:der-of-rdpf}).
 Therefore $\der{G}$ is finite if and only if $G_i$ is abelian for all but finitely many~$i$.
\end{proof}

\end{section}

\begin{section}{Bounded approximate identities in maximal ideals of $\Zl^1(G)$}\label{s:BAI-in-maxideal}
It is well known that if $A$ is an amenable, unital, commutative Banach algebra, then each maximal ideal of $A$ has a bounded approximate identity. (See, for instance, \cite[Lemma 2.3.6]{RunLec}.) Moreover, the proof shows that the norms of these bounded approximate identities are bounded from above by $1+\AM(A)$.

When $G$ is an FC group, $\Zl^1(G)$ has been studied by Stegmeir in~\cite{Steg}, where he gives an example in which $\Zl^1(G)$ has a maximal ideal without a bounded approximate identity, and is therefore non-amenable.
This example (see Example~\ref{eg:stegmeir} below) is in fact an RDPF group. We are therefore motivated to address a more general question: if $(G_i)_{i\in\Ind}$ is a family of finite groups, $G$ is their restricted direct product, and $\psi\in\Sp(\Zl^1(G))$, when does $\psi$ have a bounded approximate identity? This will be answered by Theorem~\ref{320} below.

First, we need an explicit description of $\Sp(\Zl^1(G))$ in terms of the individual groups~$G_i$.
This is given by the following theorem, which is essentially known
% referee's choice of words!
but which we state explicitly for convenience and clarity.

\begin{thm}\label{t:spectrum-of-Zl^1(G)}
Let $(G_i)_{i \in \Ind}$ be a family of finite groups and $G$ their restricted direct product.
Then there is a homeomorphism from $\Sp(\Zl^1(G))$ onto
\[
\prod_{i \in \Ind} \Sp(\Zl^1(G_i))\defeq\{  (\psi_i)_{i\in \Ind}: \psi_i \in \Sp(\Zl^1(G_i))\;\forall\ i \in \Ind\}
\]
equipped with the product topology.
In particular, $\Sp(\Zl^1(G))$ is totally disconnected.
\end{thm}

The short proof given below is due to the referee, and makes use of the general results of Liukkonnen and Mosak \cite[\S1]{LiMos} for locally compact \FCbar\ groups. One could instead give a direct but lengthier proof, using only Lemma~\ref{l:spectrum-for-finite-groups} and the definition of the Gelfand topology. (Note that the exposition given in \cite[\S1]{LiMos} is not self-contained, and rests on previous papers of Hulanicki and Mosak.)

\begin{proof}
Let $\cC$ be the convex set consisting of all positive definite, conjugation
invariant functions $\psi$ on G with $\psi(e)=1$, and let $\Ch(G)$ denote the set
of extreme points of $\cC$. By \cite[Proposition 1.1 and Corollary 1.3]{LiMos}, $\Sp(\Zl^1(G))$ is canonically homeomorphic to $\Ch(G)$ when we equip the latter space with the topology of pointwise convergence. Therefore it suffices to
show that $\Ch(G)=\prod_{i\in I} \Ch(G_i)$.

It is clear that for each $\psi\in \Ch(G)$, $\psi|_{G_i}\in \Ch(G_i)$ and the mapping $\Psi: \Ch(G) \to \prod_{i\in I} \Ch(G_i)$, $\psi \mapsto {\psi|_{G_i}}$ is continuous. Moreover, $\Psi$ is surjective, since its image contains all finite products. Hence it is a homeomorphism since $\Ch(G)$ is compact.
\end{proof}

%We need some preliminary observations, which all follow from basic properties of the nonabelian Fourier transform for finite groups.
Let $G$ be a finite group. As observed in the proof of Lemma~\ref{l:spectrum-for-finite-groups}, the Gelfand transform maps  $\Zl^1(G)$ isomorphically onto $\Cplx^{\widehat{G}}$, and an explicit formula for the inverse of the Gelfand transform is
\begin{equation}
(a_\pi)_{\pi\in\widehat{G}} \mapsto \frac{1}{|G|} \sum_{\pi\in\widehat{G}} a_\pi d_\pi \Tr\pi
\end{equation}

Let $\psi_\sigma$ be the algebra character on $\Zl^1(G)$ that corresponds to the irreducible group representation $\sigma$. The Gelfand transform maps $\ker\psi_\sigma$ bijectively onto $\{ (a_\pi)_{\pi\in\widehat{G}} \colon a_\sigma=0\}$. Since the latter has an obvious identity element, namely $(a_\pi)_{\pi\in\widehat{G}}$ where
$a_\pi = 0$ if $\pi=\sigma$ and $1$ otherwise,
we see that $\ker\psi_\sigma$ has an identity element. Denoting this identity element by $u_\sigma$, we have
\begin{equation}\label{eq:identity-in-ideal}
 \delta_e - u_\sigma = \frac{1}{|G|}d_\sigma \Tr\sigma
\end{equation}

Now suppose $G=G_1\times \dotsb G_n$, in which case we can identify $\ell^1(G)$ (via an isometric isomorphism of Banach algebras) with the $n$-fold projective tensor product $\ell^1(G_1)\ptp \dots \ptp \ell^1(G_n)$.
Let $\sigma\in\widehat{G}$: then for $i=1,\dots, n$ there exists $\sigma_i\in \widehat{G_i}$ such that
$\sigma = \sigma_1 \times \dots \times \sigma_n$;
and since $d_\sigma = \prod_{i=1}^n d_{\sigma_i}$, Equation~\eqref{eq:identity-in-ideal} becomes
\begin{equation}\label{eq:id-for-kernel-in-product}
\delta_e - u_\sigma =
\left( \frac{1}{|G_1|} d_{\sigma_1} \Tr\sigma_1 \right)
\otimes\dots\otimes
\left( \frac{1}{|G_n|} d_{\sigma_n} \Tr\sigma_n \right)
 \in \ell^1(G_1)\ptp \dots \ptp \ell^1(G_n) .
\end{equation}

%%
%% ONE OF OUR "BIG" THEOREMS FOR THIS PAPER
%%

\begin{thm}\label{320}
Let $(G_i)_{i\in\Ind}$ be a family of finite groups, and let $G=\rdp{i}{\Ind}{G}$.  Let $\omega \in \Sp(\Zl^1(G))$, and let $(\chi_i)_{i\in\Ind}$ be the corresponding family of unnormalized characters.
Then $\ker\omega$ has a bounded approximate identity, if and only if $d_{\chi_i}\|\chi_i\|_1 = |G_i|$ for all but finitely many $i\in\Ind$.
\end{thm}

\begin{proof}
We start by setting up some temporary notation. Define  for each $F\subseteq \Ind$ finite, $H_F=\bigoplus_{i\in F} G_i$ and
$H_F^c =\bigoplus_{i\in\Ind\setminus F} G_i$, so that $\Zl^1(G)=\Zl^1(H_F\times H_F^c)$.
Write $E_F^c$ for the identity element of $\Zl^1(H_F^c)$, i.e.~the unit point mass concentrated at the identity element of $H_F^c$, and likewise write $E_F$ for the identity element of $\Zl^1(H_F)$.
We also write $d_i$ for the degree of~$\chi_i$.

First suppose that $d_i \norm{\chi_i}_1=|G_i|$ for all but finitely many $i\in\Ind$.
As in Equation~\eqref{eq:id-for-kernel-in-product}, define $u_F\in \Zl^1(H_F)$ by
\[ u_F = E_F - \bigotimes_{i\in F} \frac{1}{|G_i|} d_i \chi_i\,. \]

Let $\om_F \defeq \om\imath_F$, where $\imath_F: \Zl^1(H_F)\to \Zl^1(G)$ was defined in \eqref{eq:inclusion-HM}.
Then $u_F$ is the identity element of $\ker\omega_F$. As $\ptp$ is a cross-norm,
\[ \sup_F \norm{u_F}_1 \leq 1 +  \sup_F \norm{ \bigotimes_{i\in F} \frac{1}{|G_i|} d_i \chi_i }
 = 1 + \sup_F \prod_{i\in F} \left(\frac{1}{|G_i|} d_i \norm{\chi_i}_1\right)
< \infty\,.\]
%%%
Moreover, since $\norm{ \imath_F(u_F) }_1= \norm{u_F}_1$\/, the family $(\imath_F(u_F))_{F\subset\Ind,|F|<\infty}$ is bounded. We claim that it is, when ordered by inclusion of finite subsets, a bounded approximate identity for $\ker\omega$.
To prove this, it will suffice to prove that
\begin{equation}\label{eq:it-is-BAI}
\lim_{F\subseteq\Ind, |F|<\infty} (\imath_F(u_F)) * f = f \quad\text{for all $f\in \ker\omega$;}
\end{equation}
and by a standard approximation argument, we may assume without loss of generality that $f$ has compact (i.e.~finite) support.
Thus, given $f \in Zc_c(G) \cap \ker\omega$, let $S$ be the support of $f$; if $F$ is any finite subset of $\Ind$ containing $S$, then
\[ f=\imath_S\Proj_S(f)=\imath_F\Proj_F(f) \]
(where $\Proj_S$ and $\Proj_F$ are the homomorphisms defined in~\eqref{eq:projection-HM}).
Moreover, $0=\omega(f)=\omega_F(\Proj_F(f))$.
Consequently $\Proj_F(f)\in\ker\omega_F$, and thus
\[ f*\imath_F(u_F) = \imath_F(\Proj_F(f)*u_F) = \imath_F(\Proj_F(f))= f \]
for all finite $F$ containing $S$. This proves Equation~\eqref{eq:it-is-BAI}, as required.

Conversely, suppose that $\ker\omega$ has a bounded approximate identity say $(h_\alpha)_\alpha$. For each $F\subseteq \Ind$ when $|F|<\infty$ define
\[
 \Lambda_F^\omega(f\otimes g)=\omega_{F^c}(g)f
\quad\text{for all $f\in \Zl^1(H_F)$ and $g\in \Zl^1(H_F^c)$.} \]
Since $\omega_{F^c}=\omega\imath_{F^c}$ has norm $1$, being an algebra character, we have
\[
\|\Lambda_F^\omega(f\otimes g)\|_1\leq \norm{f}_1 \norm{g}_1,
\]
%% see \cite[p 235]{BonDun}.
and so $\Lambda^\omega_F$ defines a linear contraction from $\Zl^1(H_F\times H_F^c)$ onto $\Zl^1(H_F)$, using Lemma~\ref{l:zl1-of-binary-product}.

Moreover, given $f_1,f_2\in \Zl^1(H_F)$ and $g_1,g_2\in \Zl^1(H_F^c)$, it is easily checked that
\[
   \Lambda_F^\omega\big((f_1\otimes g_1)*(f_2\otimes g_2)\big)
 = \Lambda_F^\omega(f_1\otimes g_1) * \Lambda_F^\omega (f_2\otimes g_2);
\]
hence, by linearity and continuity, $\Lambda_F^\omega$ is an algebra homomorphism.

Observe, since $\om_F\Lambda_F^\omega = \omega$,
 that $\Lambda_F^\omega(\ker\omega)\subseteq \ker\omega_F$.
Moreover, for each $f\in \ker\omega_F$, $\imath_F(f)\in\ker \omega$, and $\Lambda_F^\omega\imath_F(f)=f$. Since  $u_F\in\ker\omega_F$ and $(h_\alpha)$ is a bounded approximate identity for $\ker\omega_F$\/,
\[ \begin{aligned}
\|\Lambda_F^\omega(h_\alpha) - u_F\|_1
 & = \|\Lambda_F^\omega(h_\alpha)*u_F - u_F\|_1  \\
 & = \|\Lambda_F^\omega(h_\alpha)*\Lambda_F^\omega\imath_F(u_F)-\Lambda_F^\omega\imath_F(u_F)\|_1 \\
 & = \|\Lambda_F^\omega \left(h_\alpha* \imath_F(u_F)-\imath_F(u_F)\right)\|_1
 & \rightarrow 0\;.
\end{aligned} \]
Because $\|\Lambda_F^\omega\|\leq 1$,
\[
\sup_{F\subseteq\Ind,|F|<\infty} \sup_\alpha\|\Lambda_F^\omega(h_\alpha)\|_1
 \leq \sup_\alpha\|h_\alpha\|_1 \leq M
\]
 for some $M>0$, thus $\|u_F\|_1\leq M$ for all finite subsets $F\subseteq\Ind$. Hence
\begin{equation}\label{eq:beethoven}
\prod_{i\in F} \frac{d_i}{|G_i|} \| \chi_i\|_1 = \norm{ E_F- u_F }_1 \leq M+1
\end{equation}

% Heavily edited by YC.
Let $i\in\Ind$. For each $i$, $|G_i|^{-1} d_i  \chi_i$ is a central idempotent in the group algebra $\ell^1(G_i)$; in particular it has $\ell^1$-norm $\geq 1$. Moreover, by Rider's theorem (Theorem \ref{t:rider}),
\[ \text{either}\quad
 |G_i|^{-1}d_i \norm{\chi_i}_1 = 1
\qquad\text{or}\quad
 |G_i|^{-1}d_i \norm{\chi_i}_1 \geq \frac{301}{300}\,.
\]
It therefore follows from \eqref{eq:beethoven} that
\[ | \{ i \in\Ind \st  d_i \norm{\chi_i}_1 > \abs{G_i}\} | \leq \frac{\log (M+1 )}{\log 301 - \log 300} < \infty \,.\]
In particular, $d_i \norm{\chi_i}_1 =|G_i|$ for all but finitely many~$i$.
\end{proof}

%%% FOLLOWING OBSERVATION BY YC

\begin{thm}\label{t:every-max-ideal-has-BAI}
Let $(G_i)_{i\in\Ind}$ be a family of finite groups and let $G=\rdp{i}{\Ind}{G}$. Then the following are equivalent:
\begin{numlist}
\item\label{li:one} every maximal ideal in $\Zl^1(G)$ has a bounded approximate identity;
%% \item\label{li:two} there is a finite subset $X\subset \Sp(\Zl^1(G))$ such that for each $\omega\in \Sp(\Zl^1(G))\setminus X$, $\ker(\omega)$ has a bounded approximate identity;
\item\label{li:three} there is a finite subset $F\subset\Ind$ such that, for each $i\in\Ind\setminus F$ and each irreducible group character $\chi$ of $G_i$, we have $d_\chi \norm{\chi}_1 = |G_i|$.
\item\label{li:four} there exists a constant $M>0$ such that each maximal ideal in $\Zl^1(G)$ has a bounded approximate identity of norm $\leq M$.
\end{numlist}
\end{thm}

\begin{proof}\

\para{\ref{li:four}$\implies$\ref{li:one}}
This is trivial.

\para{\ref{li:three}$\implies$\ref{li:four}}
Let $F$ be as assumed in \ref{li:three}, and define
\[
M\defeq \prod_{i \in F} \sup_{\pi\in \widehat{G_i}} \frac{d_\pi}{|G_i|} \norm{\Tr\pi}_1 < \infty.
\]
Given $\omega\in\Sp(\Zl^1(G))$, let $(\chi_i)$ be the corresponding family of (irreducible) group characters, and let $d_i$ denote the degree of $\chi_i$. For each finite subset $T\subset\Ind$, define $u_T\in\Zl^1(\rdp{i}{T}{G})= \bigptp_{i\in T} \Zl^1(G_i)$ by
\[ u_T = \delta_e - \bigotimes_{i\in T}  \frac{d_i}{|G_i|} \chi_i \]
Order the net $(\imath_T(u_T))$, where $T$ ranges over all finite subsets of $\Ind$, by inclusion. Then by an argument like that in the proof of Theorem~\ref{320}, $(\imath_T(u_T))$ is a bounded approximate identity for $\ker\omega$, with $\sup_T \norm{\imath_T(u_T)} \leq M+1$.

\para{\ref{li:one}$\implies$\ref{li:three}}
We prove the contrapositive. Suppose that \ref{li:three} does not hold. Then there exists an infinite set $S\subset \Ind$, and for each $j\in S$ an irreducible group character $\phi_j$ on $G_j$, such that $d_j\norm{\phi_j}_1\neq |G_j|$.
Since $|G_j|^{-1} d_j\phi_j$ is an idempotent in $\Zl^1(G_j)$, the result of Rider implies that $|G_j| d_j\norm{\phi_j}_1 \geq 301/300$.
Now let $\omega$ in $\Sp(\Zl^1(G))$ be such that the corresponding family $(\chi_i)$ of group characters satisfies $\chi_j= \phi_j$ for all $j\in\Ind$.
Then as in the proof of Theorem~\ref{320}, we can show that $\ker(\omega)$ does not have a bounded approximate identity.
\end{proof}

We have already observed that, for any irreducible group character $\chi$ on a finite group $G$, we have the lower bound $d_\chi \norm{\chi}_1 \geq |G|$. Let us examine this inequality more closely.

\begin{lem}\label{l:1-minimal}
Let $G$ be a finite group and $\chi$ an irreducible character on $G$. Then $d_\chi \norm{\chi}_1 \geq |G|$. Moreover, equality holds if and only if
\begin{equation}\label{eq:ramified}
\abs{\chi(x)} \in \{0, d_\chi\} \quad\text{for all $x\in G$.}
\end{equation}
\end{lem}

%% The lower bound on $\norm{\chi_\pi}$ follows from the fact, already used above, that $d_\pi\chi_\pi$ is an idempotent in $\ell^1(G)$.

\begin{proof}
Since $\chi$ is irreducible, $\sum_{x\in G} \abs{\chi(x)}^2 = |G|$. Hence the first statement in the lemma follows from the trivial inequality
\begin{equation}\label{eq:easy}
  \sum_{x\in G} \abs{\chi(x)}^2 \leq d_\chi \sum_{x\in G} \abs{\chi(x)}\,.
\tag{$*$}
\end{equation}
For the second statement, we need to show that equality holds in \eqref{eq:easy} if and only if \eqref{eq:ramified} is satisfied. The `if' direction is clear. Conversely, if \eqref{eq:ramified} is not satisfied, pick $y\in G$ such that $0< \abs{\chi(y)} < d_\chi$. Then $\abs{\chi(y)}^2 < d_\chi \abs{\chi(y)}$, so that
%% [{\bf I changed some typos here} -MA.]
\[\sum_{x\in G} \abs{\chi(x)}^2 = \abs{\chi(y)}^2 + \sum_{x\in G\setminus\{y\}} \abs{\chi(x)}^2 < d_\chi\abs{\chi(y)} +  d_\chi \sum_{x\in G\setminus\{y\}} \abs{\chi(x)} =
d_\chi \sum_{x\in G} \abs{\chi(x)}\]
as required.
\end{proof}

Condition~\eqref{eq:ramified} is very restrictive, and is related to some concepts from the character theory of finite groups which we now briefly describe.

\begin{dfn}[The centre of a group character]\label{d:Z-of-character}
Let $\chi$ be an irreducible group character on a finite group $G$, of degree $d$ say. Let $\pi$ be any representation which affords $\chi$, then
\[ \{ x\in G \st \abs{\chi(x)} = d \} = \{ x\in G \st \pi(x)\in\Cplx I_d\} \supseteq Z(G) \]
Thus the set $\{x\in G \st \abs{\chi(x)} = d_\chi\}$ is a normal subgroup of $G$, usually denoted by $\bZ(\chi)$ and called the \dt{centre of $\chi$}.
\end{dfn}

Condition \eqref{eq:ramified} therefore says that $\chi$ is induced from a linear character on its centre. Following \cite{Kan_MPC97}, we say that $\chi$ is an  \dt{absolutely idempotent character}, and that a group $G$ is \dt{\groupprop} if each irreducible character of $G$ is \charprop.
It follows easily from the definition that quotients of \groupprop\  groups and products of \groupprop\  groups are also \groupprop\ .

Any discussion of hereditary properties implicitly presupposes that we can find some non-abelian examples, for otherwise the discussion would be rather vacuous. The smallest such example is the dihedral group of order~$8$, whose character table is shown in Figure~\ref{fig:chartable_D4}.
%%% have cut the stuff on extra-special groups
%%% see cut-extra-special
More examples are provided by the following result.

\begin{figure}[hpt]
\begin{center}
\begin{tabular}{|l|c|c|c|c|c|}
  \hline
  % after \\: \hline or \cline{col1-col2} \cline{col3-col4} ...
    &  $\{1\}$ & $\{s,r^2s\}$ & $\{rs,r^3s\}$ & $\{r,r^3\}$ & $\{r^2\}$\\ \hline
%   cardinal of class & & $1$ & $\nu$ & $\nu$ & $2$ & $1$\\ \hline
  $\chi_0$	& $1$ & \hfil$1$ & \hfil$1$ & \hfil$1$ & \hfil$1$\\
  $\chi_1$	& $1$ & $-1$	 &  $-1$    & \hfil$1$ & \hfil$1$ \\
  $\chi_2$	& $1$ & \hfil$1$ &  $-1$    &   $-1$   & \hfil$1$\\
  $\chi_3$	& $1$ & $-1$	 & \hfil$1$ &   $-1$   & \hfil$1$ \\
  $\chi_\sigma$ & $2$ & \hfil$0$ & \hfil$0$ & \hfil$0$ &  $-2$\\
  \hline
\end{tabular}
\end{center}
\caption{Character table of $D_4=\langle r,s \mid sr = r^3s\rangle$}
\label{fig:chartable_D4}
\end{figure}

\begin{thm}\label{t:MO_Isaacs}
Let $G$ be a finite group which is nilpotent of class $2$. Then $G$ is \groupprop.
\end{thm}

This result was originally communicated to the second author by I. M. Isaacs, and for sake of completeness we include a paraphrased version of his argument. (However, see Remark~\ref{r:kaniuthed} below.)

\begin{proof}[{Proof (Isaacs, \cite{MO_Isaacs})}]
Let $\chi$ be an irreducible group character on $G$. It suffices to prove that $\chi$ is \charprop.
Since $G/Z(G)$ is abelian and $Z(G) \subseteq \bZ(\chi)$, $G/\bZ(\chi)$ is abelian. Therefore, by the proof of \cite[Theorem 2.31]{Isaacs_CTbook}, $\chi$ vanishes on $G\setminus \bZ(\chi)$, and hence is \charprop\ as required.
\end{proof}

\begin{rem}\label{r:kaniuthed}
There exist infinite, discrete, nilpotent \groupprop\ groups. For example, it is observed
in the discussion following \cite[Lemma 1.2]{Kan_MPC97} that all divisible or $2$-step discrete nilpotent groups are \groupprop. For further details on the relevance of the AIC condition to harmonic analysis on discrete nilpotent groups, see~\cite{Kan_JFA06} and the references therein.
% Nonetheless, we state it for finite groups since that is all we need.
\end{rem}

The authors are unaware of any intrinsic characterization of non-abelian, finite \groupprop\ groups. We can at least rule out many groups using the following result, shown to us by F.~Ladisch.

%% slight diff from submitted version
%% add into resub
\begin{thm}[{Ladisch, \cite{MO_Ladisch}}]\label{t:MO_Ladisch}
\label{t:ladisch}
Every finite \groupprop\ group is nilpotent.
\end{thm}

For sake of interest and completeness, we have included Ladisch's proof in the appendix.
We finish this section by combining the previous theorems into the following result.

\begin{thm}
Let $(G_i)$ be a family of finite groups and let $G=\rdp{i}{\Ind}{G}$.
\begin{itemize}
\item[{\rm(a)}] If every maximal ideal of
$\Zl^1(G)$ has a bounded approximate identity, then all but finitely many $G_i$ must be nilpotent.
\item[{\rm(b)}] If all but finitely many $G_i$ are $2$-step nilpotent, then there is a constant $M>0$ such that every maximal ideal of $\Zl^1(G)$ has a bounded approximate identity of norm $\leq M$.
\end{itemize}
\end{thm}

\begin{proof}
Combine Theorem \ref{t:every-max-ideal-has-BAI}, Lemma \ref{l:1-minimal}, Theorem~\ref{t:MO_Isaacs} and Theorem~\ref{t:MO_Ladisch}.
\end{proof}

\end{section}

\begin{section}{$c_0$ and $\ell^p$ estimates for elements of $\Sp(\Zl^1(G))$}\label{s:c0-lp}

\begin{subsection}{Preliminary discussions}
We start in the more general setting of a discrete FC group $G$. 
Given an algebra character $\varphi$ on $\Zl^1(G)$, there is an obvious way to define a corresponding function $\tilphi:G\to\Cplx$ which is constant on conjugacy classes (cf.~the proof of Theorem~\ref{t:spectrum-of-Zl^1(G)}).
Abusing terminology, we will therefore say that $\varphi$ \dt{belongs to~$c_0$}, or \dt{to~$\ell^p$} for some $1\leq p<\infty$, when the corresponding function $\tilphi$ is in $c_0(G)$ or $\ell^p(G)$ respectively.
We can now state the following result of Stegmeir.

\begin{lem}[Stegmeir, {\cite[Lemma 3]{Steg}}]
\label{l:stegmeir}
Let $G$ be a discrete FC group and let $\om\in \Sp(\Zl^1(G))$. Suppose that $\ker\om$ has a bounded approximate identity and $\om\in c_0$. Then $\om\in\ell^2$.
\end{lem}

The proof in \cite{Steg} is a somewhat indirect argument by contradiction, using the existence and basic properties of a Plancherel measure on the maximal ideal space of $\Zl^1(G)$.
One motivation for the present work was to obtain a more concrete approach in the more restricted setting of restricted direct products of finite groups.

Stegmeir originally applied Lemma~\ref{l:stegmeir} to show that for a certain FC group $G$, constructed explicitly in \cite{Steg}, $\Zl^1(G)$ has a maximal ideal without a bounded approximate identity. We will present his example later, and -- since it is in fact an RDPF group -- we will use Theorem~\ref{320} to obtain a direct proof which does not require Lemma~\ref{l:stegmeir}.
(En route, we will see that for RDPF groups, Lemma~\ref{l:stegmeir} is in fact not very useful, see Theorem~\ref{t:busting-stegmeir} and the remarks preceding~it.)

In view of Lemma~\ref{l:stegmeir}, it is natural to attempt to classify the characters on $\Zl^1(G)$ which lie in $c_0(G)$, and those which lie in $\ell^2(G)$. We shall provide partial results in this direction, restricting attention to the cases where $G=\rdp{i}{\Ind}{G}$ is an RDPF group. For such groups, we can make the correspondence between algebra characters on $\Zl^1(G)$ and certain functions on $G$ completely explicit. Namely,
given $\omega\in \Sp(\Zl^1(G))$, let $(\chi_i)_{i\in\Ind}$ be the corresponding family of irreducible group characters, and let $\psi$ be the function on $G$ corresponding to $\omega$; then for each $x=(x_i)_{i\in\Ind}\in G$ we have
\begin{equation}\label{eq:spectrum-element}
\tilom(x)=\prod_{i\in\Ind} d_i^{-1}\chi_i(x_i)
\end{equation}
where $d_i$ is the degree of $\chi_i$ for each $i\in\Ind$. (The product is well-defined, because $d_i^{-1}\chi_i(x_i)=d_i^{-1}\chi_i(e_{G_i})=1$ for all but finitely many $i\in\Ind$.)
Note that since the family $(d_i^{-1}\chi_i)_{i\in\Ind}$ satisfies the conditions of Lemma~\ref{l:products-of-functionals}, we can use that lemma to calculate $\norm{\tilom}_p$ in terms of the group characters $(\chi_i)$.

\end{subsection}

\begin{subsection}{Characterizing when $\omega\in c_0$}
\begin{dfn}\label{d:peaking}
Let $\chi$ be a group character on a finite group $G$. The \dt{maximal character ratio} of $\chi$, denoted by $\mcr(\chi)$, is defined to be
\[ \sup\{ d_\chi^{-1}\abs{\chi(g)} \st g\in G \setminus\{e\}\}\]
and is clearly a real number in $[0,1]$.
\end{dfn}

%% Note that since $\abs{\chi(g)}=d_\chi$ whenever $g\in Z(G)$ and $\chi$ is irreducible, we always have $\mcr(\chi)=1$ whenever $G$ has non-trivial centre.

\begin{rem}\label{r:properties-of-mcr}
We note some basic properties of the maximal character ratio, whose proofs are easy and are therefore omitted.
\begin{numlist}
\item If $\chi$ is irreducible, then $\mcr(\chi)>0$.
\item Let $\chi$ be a group character on $H$ and $\psi$ a group character on $K$. Then $\chi\otimes\psi$ is a group character on $H\times K$, and $\mcr(\chi\otimes\psi)=\max(\mcr(\chi),\mcr(\psi))$.
\end{numlist}
\end{rem}

\begin{thm}\label{t:in-c0}
Let $G=\rdp{i}{\Ind}{G}$, let $\om\in \Sp(\Zl^1(G))$, and let $(\chi_i)$ be the corresponding family of group characters. Then $\om\in c_0$ if and only if $(\mcr(\chi_i))_{i\in\Ind} \in c_0(\Ind)$.
\end{thm}

\begin{proof}
%% \YCrem{following proof given by MA, rewritten by YC}
Suppose that $\omega \in c_0$. Let $\veps>0$, and choose a finite subset $S\subseteq G$ such that $\abs{\tilom(x)}<\veps$ for all $x\in G\setminus S$. Let $F\subseteq \Ind$ be a finite subset satisfying
\[ S\subseteq \rdp{i}{F}{G} \times \bigoplus_{i\in \Ind\setminus F} \{e_{G_i}\} \;. \]
Let $j\in \Ind \setminus F$; then for each $y\in G_j\setminus\{e_{G_j}\}$ we have
$d_j^{-1}|\chi_j(y)|=|\tilom(\imath_j(y))|<\veps$, so that $\mcr(\chi_j) < \veps$.
Thus $(\mcr(\chi_j))_{j\in F}\in c_0$\/.

Conversely, suppose that $(\mcr(\chi_j))_{j\in F}\in c_0$\/. Let $\veps>0$; then by hypothesis there exist a finite subset $F\subseteq \Ind$ such that
\begin{equation}\label{eq:bopara}
|d_i^{-1}\chi_i(y)|<\veps  \quad\text{for each $j\in\Ind\setminus F$ and each $y\in G_j\setminus\{e_{G_j}\}$.}
\end{equation}
Define $S=\{ x=(x_i) \in G \st x_j=e_{G_j} \;\text{for all $j\in \Ind\setminus F$} \}$, which is a finite subset of~$G$. Let $x\in G\setminus S$; then there exists $j\in\Ind\setminus F$ such that $x_j \in G_j \setminus \{e_{G_j}\}$, and so by \eqref{eq:bopara} we have
\[
\abs{\tilom(x)}=\prod_{i\in \Ind} d_i^{-1}\abs{\chi_i(x_i)} \leq d_j^{-1}\abs{\chi_j(x_j)} <\veps.
\]
Hence $\om\in c_0$.
\end{proof}

\begin{eg} Let $G=\rdp{i}{\Ind}{G}$ and let $\om\in\Sp\Zl^1(G)$, with $(\chi_i)$ being the corresponding family of group characters.
\begin{numlist}
\item Suppose that $G_i$ is abelian for infinitely many~$i$. For each $i$ such that $G_i$ is abelian, $\chi_i$ is linear and so $\mcr(\chi_i)=1$. Therefore, by Theorem~\ref{t:in-c0}, $\om\notin c_0$.
\item Suppose there is a fixed finite non-abelian group $K$ such that $G_i=K$ for infinitely many~$i$. Define $m$ to be the minimum value of $\mcr(\chi)$ as $\chi$ runs over all irreducible group characters on $K$; then $m>0$ (see Remark~\ref{r:properties-of-mcr}), and $\mcr(\chi_i)\geq m$ for infinitely many~$i$. So once again, we know by Theorem~\ref{t:in-c0} that $\om\notin c_0$.
\end{numlist}
\end{eg}

To obtain algebra characters on $\Zl^1(G)$ which \emph{do} lie in $c_0$, we need to have examples of group characters whose maximal character ratio can be arbitrarily small.
We shall describe two families of such examples.

\begin{eg}
\label{eg:steinberg}
Let $q=2^n$ for some $n\geq 2$, and let $\bbF_q$ be the finite field of order~$q$.
Consider $\SL(2,q)$, the special linear group over $\bbF_q$; this is known to be simple. It acts by projective transformations on the projective line over $\bbF_q$, and hence has a (transitive) permutation representation on the set $\bbF_q\cup\{\infty\}$.

The character $\chi$ of this permutation representation satisfies $\norm{\chi}_2^2= 2\abs{\SL(2,q)}$.
 Hence, by subtracting the augmentation character from $\chi$, we obtain an irreducible character $\St_q$ that has degree~$q$.
($\St_q$ is called the \dt{Steinberg character} of $\SL(2,q)$.)
By consulting known tables (see, e.g.~\cite[Ch.~28, Exercise 2]{JamLie}), or examining the permutation representation to work out the values taken by $\St_q$, we find that $\mcr(\St_q)=1/q$.

Hence, if we take any increasing sequence of positive integers $(n_i)_{i\in \Nat}$,
and let $q_i=2^{n_i}$, $G=\bigoplus_{i\in \Nat} \SL(2,q_i)$, and $\omega \in \Sp(\Zl^1(G))$ the
algebra character corresponding to the family $(\St_{q_i})_{i\in \Nat}$, then $\omega$ lies in $c_0$.

(One could replace $q$ with any prime power $\geq 4$; we still obtain an irreducible character $\St_q$ on $\SL(2,q)$, but since $\SL(2,q)$ now has non-trivial centre we will have $\mcr(\St)=1$; on the other hand, $\St$ descends to a character on $\PSL(2,q)$ which does have maximal character ratio~$1/q$. We omit the details.)
\end{eg}

\begin{eg}[Affine groups of finite fields]
\label{eg:aff-of-finfield}
Let $\bbF_q$ be a finite field of order $q$. The \dt{affine group of $\bbF_q$}, which we shall denote by $\Aff(q)$, is defined to be the set
\[ \left\{ \affine{a}{b} \st a \in\bbF_q^\times, b\in \bbF_q \right\} \]
equipped with the group structure it inherits from the usual matrix product and inversion.
It is a metabelian group; more precisely, it is isomorphic to the semidirect product $\bbF_q \rtimes \bbF_q^\times$\/.

The character table of $\Aff(q)$ is well known (see~\cite[Chapter~16, Table~II.3]{te}, for instance) and is shown in Figure~\ref{fig:affine-over-Fq}.
Observe that $\Aff(q)$, which has order $q(q-1)$, has $q-1$ linear characters (corresponding to those on the quotient group~$\bbF_q^\times$), and precisely one non-linear character, which has degree $q-1$ and maximal character ratio~$1/(q-1)$.

\end{eg}

\begin{figure}[hpt]
\begin{center}
\begin{tabular}{|l|c|c|c|}
  \hline
  % after \\: \hline or \cline{col1-col2} \cline{col3-col4} ...
    & $\{e\}$ &  $C_1$  & $C_{(0,y)}$  ($y \in 2,\cdots,q-1$) \\ \hline
   size of conjugacy class & $1$ & $q-1$ & $q$ \\ \hline
  $\chi_1$ &  $1$ & $1$ & $1$ \\
  $\chi_j$  ($ j\in 2,\cdots,p-1$) & $1$ & $1$ & $\theta_j((0,y))$ \\
  $\chi_\pi$ &  $q-1$ & $-1$ & $0$ \\
  \hline
\end{tabular}
\end{center}

\caption{Character table for $\Aff(q)$}
\label{fig:affine-over-Fq}
\end{figure}

Hence, if we take any increasing sequence of prime powers $(q_i)_{i\in \Nat}$,
and let $G=\bigoplus_{i\in \Nat} \Aff(q_i)$, let $\chi_{\pi_i}$ be the non-linear character
on $\Aff(q_i$), and $\omega \in \Sp(\Zl^1(G))$ the
algebra character corres\-ponding to the family $(\chi_{\pi_i})_{i\in \Nat}$, then $\omega$ lies in $c_0$.

\begin{eg}[Stegmeir]\label{eg:stegmeir}
Let $P$ denote the set of all prime numbers, and for each $p\in P$ define
$G_p \defeq \prod_{i=1}^{p-1} \Aff(p)$.
% where $\Aff(p)$ was defined earlier.
Writing $\pi$ for the unique non-linear irreducible representation of $\Aff(p)$, let
\[ \chi_p = (\Tr\pi)\otimes \dots \otimes (\Tr\pi) \qquad\text{($p$ times)} \]
which is an irreducible group character on $G_p$.

Let $G \defeq\rdp{p}{P}{G}$, and let $\omega \in \Sp(\Zl^1(G))$ be the algebra character corresponding to the family $(\chi_p)_{p\in P}$ (see Theorem \ref{t:spectrum-of-Zl^1(G)}).
Since the group character $\Tr\pi$ has maximal character ratio $(p-1)^{-1}$, so does $\chi_p$ (by Remark \ref{r:properties-of-mcr}), for each $p\in P$. Therefore, by Theorem~\ref{t:in-c0}, $\omega$ lies in $c_0$.
\end{eg}

The original point of Stegmeir's example is that the $\omega$ defined above lies in $c_0$ but not in $\ell^2$; hence, by Lemma~\ref{l:stegmeir} above, $\ker\omega$ has no bounded approximate identity.
%  Stegmeir's proof that $\omega\notin\ell^2$ was done by direct estimates, which we are able to generalize significantly; this will be done in Section~\ref{s:lp-estimates} below.
However, Lemma~\ref{l:stegmeir} turns out to be somewhat misleading when dealing with restricted direct products of finite groups.
Recall that the lemma says: ``if $\ker\omega$ has a bounded approximate identity and $\omega\in c_0(G)$, then $\omega\in\ell^2(G)$''. The next result shows that if $G$ is an RDPF group -- as in the example Stegmeir considers -- then this statement is conditioning on an empty set.

\begin{thm}\label{t:busting-stegmeir}
Let $(G_i)$ be a family of finite groups and let $G=\rdp{i}{\Ind}{G}$. Let $\om\in \Sp(\Zl^1(G))$ and suppose that $\ker\omega$ has a bounded approximate identity. Then $\omega\notin c_0$.
\end{thm}

\begin{proof}[Proof of Theorem~\ref{t:busting-stegmeir}]
Let $(\chi_i)_{i\in\Ind}$ be the family of irreducible group characters which corresponds to $\omega$, and as usual let $d_i$ denote the degree of $\chi_i$\/.
By Theorem~\ref{320} and Lemma \ref{t:every-max-ideal-has-BAI}, there is a finite set $F\subset\Ind$ such that for all $i\in\Ind\setminus F$, the function $\abs{\chi_i}: G_i \to \Cplx$ takes values in $\{0,d_i\}$.

Observe that for any finite group $H$ and any irreducible character $\psi$ on~$H$, there is some $y\in H\setminus\{e\}$ such that $\psi(y)\neq 0$
(this follows from Schur orthogonality).
Hence, from the previous paragraph, for each $i\in\Ind\setminus F$ there exists $y_i\in G_i\setminus\{e_{G_i}\}$ with $\abs{\chi_i(y_i)}=d_{\chi_i}$, so that $\mcr(\chi_i)=1$. Now apply Theorem~\ref{t:in-c0}.
 (Alternatively, one can avoid appealing to Theorem~\ref{t:in-c0}, by using the $\{y_i\}$ to directly construct an infinite sequence $\{x_k\}\in G$ satisfying $\abs{\omega(x_k)}=1$ for all~$k$.)
\end{proof}

\begin{rem}
Going back to Stegmeir's example, we can show more directly that $\ker\omega$ has no bounded approximate identity, by following the proof of Theorem~\ref{320} and using the information in Figure~\ref{fig:affine-over-Fq} to explicitly compute the $\ell^1$-norms of the group characters that make up~$\omega$. In particular, we can do without Rider's theorem (Theorem~\ref{t:rider}), which seems to be needed for the proof of Theorem~\ref{t:busting-stegmeir}.
\end{rem}

\end{subsection}

\begin{subsection}{Examples where $\om\in\ell^p$ for various $p$}\label{s:lp-estimates}
Although we are primarily interested in deciding whether or not $\om\in\ell^2$, in most of our examples the calculations can be done just as easily for $\ell^p$.

In the following lemma, infinite sums and products of real numbers are to be understood in the sense of Section~\ref{s:prelim}.

\begin{lem}\label{l:products-of-functionals}
Let $(G_i)_{i\in\Ind}$ be a family of finite groups, and let $G=\rdp{i}{\Ind}{G}$. We denote the identity element of each $G_i$ by~$e_i$.

For each $i\in\Ind$ let $\psi_i: G_i \to \Cplx$ be a function satisfying $\psi_i(e_i)=1$, and define $\Psi: G \to \Cplx$ by
\begin{equation}\label{eq:defining-product}
\Psi(x) = \prod_{i\in\Ind} \psi_i(x_i) \qquad\text{for $x=(x_i)_{i\in\Ind}$.}
\end{equation}
Then for each $p\in (0,\infty)$ we have
\begin{equation}\label{eq:norm-of-product}
\sum_{x\in G } \abs{\Psi(x)}^p = \prod_{i\in\Ind} \norm{\psi_i}_p^p\,.
\end{equation}

In particular, if $\om\in\Sp(\Zl^1(G))$ and $(\chi_i)_{i\in\Ind}$ is the corresponding family of group characters, let $d_i$ denote the degree of $\chi_i$; then
\begin{equation}\label{eq:lp-norm-of-algebra-character}
\norm{\tilom}_p = \prod_{i\in\Ind} d_i^{-1}\norm{\chi_i}_p\;.
\end{equation}
for every $p\in [1,\infty)$.
\end{lem}

\begin{proof}
First of all, we note that the product on the right hand side of \eqref{eq:defining-product} is well-defined, since $x_i=e_i$ for all but finitely many~$i$.
Also, since
$\norm{\psi_i}_p^p\geq \psi_i(e_i)=1$ for all $i\in\Ind$, the infinite product on the right-hand side of \eqref{eq:norm-of-product} is well-defined.

For each finite subset $F\subseteq\Ind$, let
\[ G_F= \{ x=(x_i) \in G \st x_i=e_i \text{ for all $i\in \Ind\setminus F$} \}, \]
which is a finite subset of $G$. We have
\begin{equation}\label{eq:dravid}
\sum_{x\in G_F}  \abs{\Psi(x)}^p
 = \sum_{x\in G_F} \prod_{i\in F} \abs{\psi(x_i)}^p
 = \prod_{i\in F} \sum_{x_i \in G_i} \abs{\psi(x_i)}^p
 = \prod_{i\in F} \norm{\psi_i}_p^p
\tag{$*$}
\end{equation}
(all sums and products being over finite indexing sets).
Note also that each finite subset of $G$ is contained in one of the form $G_F$ for some finite subset $F\subseteq\Ind$. Hence
\[ \sum_{x\in G} \abs{\Psi(x)}^p
  = \sup_{F\subseteq\Ind, |F|<\infty} \sum_{x\in G_F}  \abs{\Psi(x)}^p
\]
which, by \eqref{eq:dravid}, implies that
\[
\sum_{x\in G} \abs{\Psi(x)}^p
  = \sup_{F\subseteq\Ind, |F|<\infty}
  \prod_{i\in F} \norm{\psi_i}_p^p
= \prod_{i\in\Ind} \norm{\psi_i}_p^p
\]
which gives us the desired identity~\eqref{eq:norm-of-product}.
The last part of the lemma now follows, by using the identity~\eqref{eq:spectrum-element} and the observation that $d_i\chi_i(e_i) =1$ for all~$i$.
\end{proof}

\begin{eg}[An example using the Steinberg characters]\label{eg:more-steinberg}
Take $G=\bigoplus_{n\geq 2} \SL(2,2^n)$, and let $\om\in\Sp(\Zl^1(G))$ be the algebra character corresponding to the sequence $(\St_{2^n})_{n\geq 2}$, where $\St_{2^n}$ is the Steinberg character on $\SL(2,2^n)$ as defined in Example~\ref{eg:steinberg}.
By Theorem~\ref{t:in-c0}, $\om\in c_0$.
Consulting the known character table of $\SL(2,2^n)$, we find that
\[ \norm{\tilom}_s^s = \prod_{n=2}^\infty \frac{2^{ns}+2^{3n}-2^{2n}-2^n}{2^{ns}}. \]
From this and \eqref{eq:norm-of-product}, we see that $\om\in \ell^s$ if and only if~$s>3$.
\end{eg}

In some sense, Lemma~\ref{l:products-of-functionals} characterizes when $\om\in\ell^p$, in terms of the $\ell^p$-norms of the irreducible group characters that make up $\omega$.
On the other hand, the $\ell^p$-norms of irreducible group characters are not well understood, unless $p=2$ or the characters are linear.
In the case $p=2$ we can be slightly more specific: the following result follows immediately from the identity \eqref{eq:lp-norm-of-algebra-character}, together with the standard result that the $\ell^2$-norm of an \emph{irreducible} group character on a finite group $H$ is $|H|^{1/2}$\/.

\begin{cor}\label{c:formula-for-2norm}
Let $G=\rdp{i}{\Ind}{G}$ and let $\omega\in \Sp(\Zl^1(G))$; let $(\chi_i)_{i\in\Ind}$ be the corresponding family of (unnormalized) irreducible characters, and let $d_i$ be the degree of $\chi_i$. Then
\begin{equation}\label{eq:formula-for-2norm}
\norm{\tilom}_2 = \prod_{i\in\Ind} \frac{|G_i|^{1/2}}{d_i}
\end{equation}
\end{cor}

Equation~\eqref{eq:formula-for-2norm} imposes strong restrictions on the family $(G_i)$ if we wish to
obtain some $\om$ lying in $\ell^2$. Namely, the following must hold:
{\it for each $\delta>0 $ there exists a finite group $H$ and an irreducible character $\chi$ on $H$ which satisfies $(1+\delta)d_\chi^2 \geq |H|$}\/.
In this context, we raise the following problem, which may be of interest to researchers in finite group theory:

\para{Problem}
Given $\delta \in [0,1/2)$, classify all finite groups $H$ with the property that $(1+\delta)d_\chi^2 \geq |H|$ for some irreducible character $\chi$ on~$H$.

\medskip
The authors know of no value of $\delta$ for which such a classification is known.

\begin{eg}[Stegmeir's example, revisited]
Let $G=\rdp{p}{P}{G}$ and $\om$ be as described in Example~\ref{eg:stegmeir}. Stegmeir showed in \cite{Steg}, by direct calculation, that $\om\notin\ell^2$. This also follows easily from the formula \eqref{eq:formula-for-2norm}: for $\abs{G_p} = \abs{\Aff(p)}^{p-1} = p^{p-1}(p-1)^{p-1}$, while $d_p = (p-1)^{p-1}$, giving
\[
\norm{\tilom}_2^2 =  \prod_{p\in P} \left( 1+ \frac{1}{p-1}\right)^{p-1} \geq \prod_{p\in P} 2 = + \infty
\]

We can go further if we use Lemma~\ref{l:products-of-functionals}. Examining the character table of $\Aff(q)$, we see that for $s\in [1,\infty)$
\begin{equation}\label{eq:dernbach}
\begin{aligned}
 \left(d_p^{-1}\norm{\chi_p}_s\right)^s
 & = \left({d_{\Tr\pi}}^{-s}\norm{\Tr\pi}_s^s\right)^{p-1} \\
 & = \left( \frac{ (p-1)^s + (p-1) }{(p-1)^s}\right)^{p-1} = \left( 1+ \frac{1}{(p-1)^{s-1}}\right)^{p-1}\,.
\end{aligned}
\end{equation}
% Insert GR's argument
By the binomial expansion, the right-hand side is at least $1+(p-1)^{2-s}$. Therefore, if $1\leq s \leq 3$, we have
\[ \left(d_p^{-1}\norm{\chi_p}_s\right)^s \geq 1+ (p-1)^{-1} = \frac{p}{p-1},\]
and so
\[  \norm{\tilom}_s^s
 = \prod_{p\in P} \left( d_p^{-1}\norm{\chi_p}_s\right)^s
 \geq \prod_{p\in P} \frac{p}{p-1} = + \infty
\]
so that $\om\notin \ell^s$. On the other hand, if $s\in (3,\infty)$, then since
\[ s \log \norm{\tilom}_s
 = \sum_{p\in P} s \log (d_p^{-1}\norm{\chi_p}_s )
 = \sum_{p\in P} (p-1) \log \left( 1+ \frac{1}{(p-1)^{s-1}}\right) ,
 \]
the standard estimate $\log(1+x) < x$ for $0<x<1$ implies
\[ s \log \norm{\tilom}_s  \leq \sum_{p\in P} (p-1)^{2-s} < \sum_{n\geq 2} n^{2-s} < \infty, \]
and thus $\om\in\ell^s$.
The first author thanks G. ~Robinson \cite{MO_Robinson} for showing him these calculations.
\end{eg}

It is natural to wonder what summability properties can be exhibited by various~$\om$. Clearly, if $G$ is virtually abelian then none of the algebra characters on $\Zl^1(G)$ can lie in $c_0$. The following examples will show that within the class of metabelian groups, one can obtain examples with
 characters lying in $c_0$ but not in $\ell^s$ for $s<\infty$,
 and at the other extreme obtain examples with
 characters lying in every $\ell^s$ for $s>1$.

\begin{eg}
For each prime $p\in P$, let $G_p=\Aff(p)^{2^p}$. Let $\omega\in\Sp(\Zl^1(G))$, let $(\chi_p)_{p\in P}$ be the corresponding family of group characters, and let $d_p$ denote the degree of~$\chi_p$\/. Note that since $\Aff(p)$ has an irreducible group character with maximal character ratio $(p-1)^{-1}$, so does $G_p$ (see Remark~\ref{r:properties-of-mcr}), and therefore it is possible to have $\om\in c_0$\/.

On the other hand: let $s\in [1,\infty)$. By inspecting the possible cases, we see that
\[ \left( d_\psi^{-1} \norm{\psi}_s\right)^s
   \geq  1 + \frac{1}{(p-1)^{s-1}}
\quad\text{for every irreducible group character $\psi$ on $\Aff(p)$.}
\]
Therefore, since $\chi_p$ is a tensor product of irreducible characters on $\Aff(p)$,
\[ \left( d_p^{-1} \norm{\chi_p}_s\right)^s
	\geq \left(1 + \frac{1}{(p-1)^{s-1}}\right)^{2^p}
	 > 1+\frac{2^p}{p^{s-1}} \]
showing that $\norm{\tilom}_s^s =+\infty$.
\end{eg}

\begin{eg}
\label{eg:in-all-lp-for-p>1}
Fix a prime $p$ and consider $G=\bigoplus_{n\geq 2} \Aff(p^n)$.
%% We will show that every $\omega\in\Sp(\Zl^1(G))$ either lies outside $c_0$, or lies in $\ell^s$ for every $s>1$; moreover, both possibilities can occur.
%%
Let $\omega\in \Sp(\Zl^1(G))$, and let $(\chi_n)$ be the corresponding sequence of group characters.
Recall that $\Aff(p^n)$ has a unique non-linear irreducible character, of degree $p^n-1$, which we shall denote in this example by~$\psi_n$. For each $s\in [1,\infty)$, a short calculation gives
\[ \left( \frac{1}{p^n-1}\norm{\psi_n}_s\right)^s
 = 1 + \frac{1}{(p^n-1)^{s-1}}  \leq 1+ p^{-(s-1)(n-1)}\,.
\]
Now, for each $n$, either $\chi_n$ is linear or $\chi_n=\psi_n$\/.
Let
\[ K_\omega \defeq \{n\in\Nat: \chi_n \text{ is\ linear} \}. \]
If $K_\omega$ is infinite, then since $\mcr(\chi_n)=1$ for all $n\in K_\omega$, it follows from Theorem~\ref{t:in-c0} that $\om\notin c_0$\/.
Suppose, on the other hand, that $K_\omega$ is finite, and let $s>1$. Then by Lemma~\ref{l:products-of-functionals},
\[ \begin{aligned}
\norm{\om}_s^s = \prod_{n\in \Nat} d_{\chi_n}^{-s} \|\chi_n\|_s^s
& = \left(\prod_{n\in K_\omega} p^n(p^n-1)\right)^s \prod_{n\in\Nat\setminus K_\omega} \left(\frac{1}{p^n-1}\norm{\psi_n}_s\right)^s \\
& \leq \left(\prod_{n\in K_\omega} p^n(p^n-1)\right)^s \prod_{n\in\Nat\setminus K_\omega} \left(1+p^{-(s-1)(n-1)} \right) \\
& \leq \left(\prod_{n\in K_\omega} p^n(p^n-1)\right)^s \exp \left( \sum_{n=1}^\infty p^{-(s-1)n}\right)
\end{aligned} \]
is finite.
\end{eg}

Direct calculation shows that when $G$ is as in Example~\ref{eg:in-all-lp-for-p>1}, none of the algebra characters $\om\in\Sp(\Zl^1(G))$ lie in~$\ell^1$. The following result shows that this is not a failing of our example, but a general result.

\begin{thm}
Let $G=\rdp{i}{\Ind}{G}$ be an RDPF group, and let $\om\in\Sp(\Zl^1(G))$. If $G$ is infinite, then $\om\notin\ell^1$\/.
\end{thm}

\begin{proof}
%% \YCrem{Original proof by MA; edited by YC}
Since $G$ is infinite, $\Ind$ is infinite; we may also suppose without loss of generality that $\abs{G_i}\geq 2$ for all~$i$.
Let $(\chi_i)$ be the family of group characters corresponding to $\om$, and for ease of notation let $d_i$ denote the degree $\chi_i$.

Using Lemma~\ref{l:products-of-functionals} and the fact that $\abs{G_i}\geq d_i^2$,
\begin{equation}\label{eq:haggis}
 \|\tilom\|_1 = \prod_{i\in\Ind} d_i^{-1}\norm{\chi_i}_1 \geq \prod_{i\in \Ind} \frac{d_i}{\abs{G_i}} \norm{\chi_i}_1 \;.
\end{equation}
Define $F= \{ i\in\Ind \st d_i \norm{\chi_i}_1 = \abs{G_i}\}$. Since the support of $\chi_i$ contains at least one non-identity element, for each $i\in F$ we have $\mcr(\chi_i)=1$, by Lemma~\ref{l:1-minimal}.

If $F$ is infinite, then $(\mcr(\chi_i))_{i\in\Ind} \notin c_0(\Ind)$, so by Theorem~\ref{t:in-c0} $\om$ does not lie in $c_0$, and so cannot lie in $\ell^1$. On the other hand, if $F$ is finite, then $\Ind\setminus F$ is infinite, and so by combining the inequality \eqref{eq:haggis} with Rider's result (Theorem~\ref{t:rider}) we obtain
\[
 \|\tilom\|_1 \geq \prod_{i\in \Ind\setminus F} \frac{d_i}{\abs{G_i}} \norm{\chi_i}_1 \geq \prod_{i\in\in\Ind\setminus F} \frac{301}{300} =+\infty\;,
\]
showing that $\om\notin\ell^1$.
\end{proof}

We finish by remarking that most of our examples have been built out of groups of the form $\Aff(q)$. This is to demonstrate that we can achieve a wide range of behaviour of  $\ell^p$-norms of characters on $\Zl^1(G)$ while staying within the class of metabelian groups. However, the formulas we have to date suggest that if we merely require algebra characters that lie in $\ell^s$ where $s\in (2,\infty)$, then we should have many more examples (cf.~Example~\ref{eg:more-steinberg}).
\end{subsection}

\end{section}

\appendix

\section{Appendix: finite \groupprop\  groups are nilpotent}\label{app:ladisch}
We present a proof of Theorem~\ref{t:ladisch}, namely that a finite \groupprop\ group must be nilpotent. The argument is paraphrased from one shown to the second author by F.~Ladisch~\cite{MO_Ladisch}, and is included here with his kind permission.
We have tried to make the presentation accessible to non-specialists in character theory.

Let us recall some of the relevant definitions and basic facts.
When $\chi$ is an irreducible group character on a finite group~$G$, the centre of $\chi$, denoted by $\bZ(\chi)$ of $\chi$ is a normal subgroup of~$G$.
We say $\chi$ is \charprop\  if and only if $\bZ(\chi)=\supp(\chi)$.

The following lemma, which appears to be well known to specialists, uses no special properties of~$G$.

\begin{lem}\label{l:normal-meets-support}
Let $N$ be a normal subgroup of $G$ and let $\chi$ be an irreducible character of $G$. If $N$ contains a non-identity element, then so does $N\cap \supp(\chi)$.
\end{lem}

\begin{proof}
Consider the character $\chi\vert_N$. We have two cases to consider. If $\chi\vert_N$ is not orthogonal to the trivial character $\veps_N$, then since $N$ is normal it follows from a theorem of Clifford that $\chi$ is proportional to $\veps_N$ (see, e.g.~\cite[Corollary 6.7]{Isaacs_CTbook}). In particular, $N\cap\supp(\chi) = N$ contains a non-identity element.

On the other hand, if $\chi\vert_N$ is orthogonal to $\veps_N$, then
%\[
$0 = \pair{\chi\vert_N}{\veps_N} = \frac{1}{\abs{N}}\sum_{x\in N} \chi(x)$.
%\]
Since $\chi(e) >0$, $\chi$ must be non-zero on at least one non-identity element of~$N$.
\end{proof}

\begin{cor}\label{c:SOFA}
Let $G$ be a group with at least two elements, and let $S$ be a set of \charprop\  characters on $G$. Then $\bigcap_{\chi\in S} \bZ(\chi)$ contains a non-identity element.
\end{cor}

\begin{proof}
We induct on the size of $S$. If $S$ is empty there is nothing to prove. Otherwise, suppose $\chi_1,\dots, \chi_{n-1}$ are \charprop\  characters for which $N\defeq \bZ(\chi_1)\cap\dots\cap \bZ(\chi_{n-1})$ contains a non-identity element. $N$ is a normal subgroup of $G$, since $\bZ(\chi_i)$ is for each~$i$; and since $\chi_n$ is \charprop\ , $\supp(\chi)=\bZ(\chi_n)$. By Lemma~\ref{l:normal-meets-support}, $N\cap \bZ(\chi_n)$ therefore contains a non-identity element, completing the inductive step.
\end{proof}

\begin{proof}[Proof of Theorem~\ref{t:ladisch}]
We argue by strong induction on the order of the group. Every group of order $\leq 5$ is abelian, hence in particular both \groupprop\  and nilpotent.

Let $n\geq 6$ and suppose inductively that all \groupprop\  groups of order $<n$ are nilpotent. Let $G$ be an \groupprop\  group of order $n$. Now $Z(G)=\bigcap_{\chi\in\Irr(G)} \bZ(\chi)$ -- this is true for \emph{any} finite group, see \cite[Corollary 2.28]{Isaacs_CTbook} -- and therefore by Corollary~\ref{c:SOFA}, $Z(G)$ contains a non-identity element.
Since quotients of finite \groupprop\ groups are themselves \groupprop, $G/Z(G)$ is \groupprop\ and has order $\leq n/2 < n$, and so it is nilpotent by the inductive hypothesis. But then $G$ is a central extension of a nilpotent group, and so is itself nilpotent.
\end{proof}

\subsection*{Acknowledgements}
Some of the results here are taken from the PhD thesis of the first author (MA), who was supported by a Dean's Scholarship from the University of Saskatchewan. The second author (YC) was partially supported by NSERC Discovery Grant 402153-2011, and the third author (ES) by NSERC Discovery Grant 366066-2009.

The second author thanks F. Ladisch for useful exchanges, and I. M. Isaacs for supplying the proof of Theorem~\ref{t:MO_Isaacs}. The authors also thank the referee for several suggestions which improved the paper, in particular the proof of Lemma~\ref{l:zl1-of-binary-product}, and for bringing the prior study of absolutely idempotent characters and AIC groups to the authors' attention.

%\newcommand{\url}[1]{\href{#1}{#1}}

%\bibliography{zl1g-rdp}
%\bibliographystyle{siam}

\vfill

\contact
\end{document}